\documentclass[11pt,a4paper]{article}

\usepackage{epsf,epsfig,amsfonts,amsgen,amsmath,amstext,amsbsy,amsopn,amsthm,cases,listings,color
}
\usepackage{ebezier,eepic}
\usepackage{color}
\usepackage{multirow}
\usepackage{epstopdf}
\usepackage{graphicx}
\usepackage{pgf,tikz}
\usepackage{mathrsfs}
\usepackage[marginal]{footmisc}
\usepackage{enumitem}
\usepackage[titletoc]{appendix}
\usepackage{booktabs}
\usepackage{url}
\usepackage{mathtools}

\usepackage{pgfplots}
\usepackage{authblk}
\usepackage{amssymb}

\usepackage{wasysym}

\usepackage{empheq}

\usepackage{dsfont}

\usepackage{caption}
\usepackage{subcaption}
\pgfplotsset{compat=1.18}
\usepackage{mathrsfs}
\usepackage{wasysym} 
\usetikzlibrary{arrows}
\usepackage{aligned-overset}
\usepackage{bm}
\usepackage[backref=page]{hyperref}
\usepackage{makecell}

\allowdisplaybreaks[1]

\definecolor{uuuuuu}{rgb}{0.27,0.27,0.27}
\definecolor{sqsqsq}{rgb}{0.1255,0.1255,0.1255}

\newtheorem{definition}{Definition} [section]
\newtheorem{theorem}[definition]{Theorem}

\newtheorem{proposition}[definition]{Proposition}

\newtheorem{claim}[definition]{Claim}

\newenvironment{pf}{\noindent{\bf Proof}}

\begin{document}
\title{\bf\Large Generalized Andr\'{a}sfai--Erd\H{o}s--S\'{o}s theorems for odd cycles}
\date{\today}
\author[1]{Zian Chen\thanks{Email: \texttt{michaelchen24@163.com}}}
\author[1]{Jianfeng Hou\thanks{Research  supported by National Key R$\&$D Program of China (Grant No. 2023YFA1010202), National Natural Science Foundation of China (Grant No. 12071077) and the Central Guidance on Local Science and Technology Development Fund of Fujian Province (Grant No. 2023L3003). Email: \texttt{jfhou@fzu.edu.cn}}}
\author[1]{Caiyun Hu \thanks{Email: \texttt{hucaiyun.fzu@gmail.com}}}  
\author[2]{Xizhi Liu\thanks{Research  supported by ERC Advanced Grant 101020255. Email: \texttt{xizhi.liu.ac@gmail.com}}}
\affil[1]{Center for Discrete Mathematics,
            Fuzhou University, Fujian, 350003, China}
\affil[2]{Mathematics Institute and DIMAP,
            University of Warwick,
            Coventry, CV4 7AL, UK}
\maketitle
\begin{abstract}
    In this note, we establish Andr\'{a}sfai--Erd\H{o}s--S\'{o}s-type stability theorems for two generalized Tur\'{a}n problems involving odd cycles, both of which are extensions of the Erd{\H o}s Pentagon Problem. 
    Our results strengthen previous results by Lidick\'{y}--Murphy~\cite{LM21} and Beke--Janzer~\cite{BJ24}, while also simplifying parts of their proofs. 

\medskip

\end{abstract}
\section{Introduction}\label{SEC:Intorduction}
In this note, we consider graphs in which each vertex is uniquely labeled. 
We identify a graph with its edge set and denote the vertex set of a graph $G$ as $V(G)$. 
We use $|G|$ and $v(G)$ to denote the number of edges and the number of vertices in a graph $G$, respectively. 
For every integer $r$, we use $K_r$ and $C_r$ to denote the complete graph and cycle on $[r]$, respectively. The edge set of $C_r$ is given by $\left\{\{i,i+1\} \pmod{r} \colon i \in [r] \right\}$. 

Given two graphs $Q$ and $G$, a map $\varphi \colon V(Q) \to V(G)$ is a \textbf{homomorphism} if $\varphi(e) \in G$ for all $e \in Q$.  
If such a homomorphism exists, we say $Q$ is \textbf{$G$-colorable}. 
The collection of all homomorphisms from $Q$ to $G$ is denoted by $\mathrm{Hom}(Q,G)$, and its size is denoted by $\mathrm{hom}(Q,G)$. 
Similarly, the family of all \textbf{injective homomorphisms} from $Q$ to $G$ is denoted by $\mathrm{Inj}(Q,G)$, and its size is denoted by $\mathrm{inj}(Q,G)$.
For every vertex $v\in V(G)$, the  \textbf{$Q$-degree} of $v$ in $G$ is 
\begin{align*}
    d_{Q,G}(v)
    \coloneqq \left\{\varphi \in \mathrm{Inj}(Q,G) \colon v\in \varphi(V(Q))\right\}. 
\end{align*}
We used $\delta_{Q}(G)$ and $d_{Q}(G)$ to denote the \textbf{minimum} and \textbf{average} $Q$-degree of $G$, respectively. 

Given a family $\mathcal{F}$ of graphs, we say $G$ is \textbf{$\mathcal{F}$-free}
if it does not contain any member of $\mathcal{F}$ as a subgraph.
Given a graph $Q$ and a family $\mathcal{F}$ of graphs, let 
\begin{align*}
    \mathrm{inj}(n,Q,\mathcal{F})
    \coloneqq \max\left\{\mathrm{inj}(Q,G) \colon  \text{$v(G) = n$ and $G$ is $\mathcal{F}$-free} \right\}. 
\end{align*}
Generalizing the fundamental Tur\'{a}n problem (see~\cite{TU41}), the well-studied \textbf{generalized Tur\'{a}n number} $\mathrm{ex}(n,Q,\mathcal{F})$ (see e.g.~\cite{Erd62gen,AS16}) is defined in relation to $\mathrm{inj}(n,Q,\mathcal{F})$ through $\mathrm{ex}(n,Q,\mathcal{F}) \coloneqq {\mathrm{inj}(n,Q,\mathcal{F})}/{|\mathrm{Aut}(Q)|}$.
Here, $\mathrm{Aut}(Q) = \mathrm{Inj}(Q,Q)$ denotes the automorphism group of $Q$. 
The \textbf{generalized Tur\'{a}n density} is defined as $\pi(Q, \mathcal{F}) \coloneqq \lim_{n\to \infty} {\mathrm{ex}(n,Q,\mathcal{F})}/{\binom{n}{v(Q)}}$. 
The existence of this limit can be established by a simple averaging argument used by Katona--Nemetz--Simonovits in~\cite{KNS64}. 
%
%
To maintain consistency in notation, we will use $\mathrm{inj}(n,Q,\mathcal{F})$ in place of $\mathrm{ex}(n,Q,\mathcal{F})$ for the remainder of this note. 

In 1984, Erd\H{o}s~\cite{Erd84} conjectured that for every $n \ge 5$ the maximum number of copies of $C_5$ in an $n$-vertex triangle-free graph is attained by the balanced blowup of $C_5$. 
This conjecture was independently resolved by Grzesik~\cite{Gre12} and Hatami--Hladk\'{y}--Kr\'{a}\v{l}--Norine--Razborov~\cite{HHKNR13} for large $n$, and later by Lidick\'{y}--Pfender~\cite{LP18} for all $n$, 
utilizing the powerful Flag Algebra machinery~\cite{Raz07}. 

Two notable generalizations of Erd\H{o}s' Pentagon Problem were established recently. 
Lidick\'{y}--Murphy~\cite{LM21} determined $\mathrm{inj}(n,C_5,K_{r+1})$ for every $r \ge 3$ when $n$ is sufficiently large. 
Grzesik--Kielak~\cite{GK23} determined $\mathrm{inj}(n,C_{2r+1},\{C_3, \ldots, C_{2r-1}\})$ for every $r \ge  3$. 
Later, Beke--Janzer~\cite{BJ24} improved upon this by determining $\mathrm{inj}(n,C_{2r+1},C_{2r-1})$ for $r \ge  3$ when $n$ is large. 

We strengthen the results of Lidick\'{y}--Murphy~\cite{LM21} and Beke--Janzer~\cite{BJ24}, providing simpler proofs that also streamline parts of their original arguments. 
\begin{theorem}\label{THM:C5-Kr}
    For every integer $r \ge 3$ there exist $\delta >0$ and $N_0$ such that the following holds for every $n \ge N_0$. 
    \begin{enumerate}[label=(\roman*)]
        \item\label{THM:C5-Kr-1} If $G$ is an $n$-vertex $K_{r+1}$-free graph with 
        \begin{align}\label{equ:THM:C5-Kr}
            \delta_{C_{5}}(G) 
            \ge (1-\delta) \cdot \frac{5\cdot \mathrm{inj}(n,C_{5}, K_{r+1})}{n}, 
        \end{align}
        then $G$ is $K_{r}$-colorable. 
        \item\label{THM:C5-Kr-2} If $G$ is an $n$-vertex $C_{2r-1}$-free graph with 
        \begin{align}\label{equ:THM:C2r+1-C2r-1}
            \delta_{C_{2r+1}}(G) 
            \ge (1-\delta) \cdot \frac{(2r+1)\cdot \mathrm{inj}(n,C_{2r+1}, C_{2r-1})}{n}, 
        \end{align}
        then $G$ is $C_{2r+1}$-colorable.
    \end{enumerate}
\end{theorem}
%
%
\textbf{Remarks}.
Theorem~\ref{THM:C5-Kr} implies that for large $n$, the extremal constructions for $\mathrm{inj}(n,C_{5}, K_{r+1})$ and $\mathrm{inj}(n,C_{2r+1}, C_{2r-1})$ are $K_{r}$-colorable and $C_{2r+1}$-colorable, respectively. 
From this, the exact values of $\mathrm{inj}(n,C_{5}, K_{r+1})$ and $\mathrm{inj}(n,C_{2r+1}, C_{2r-1})$ are straightforward to determine (see~{\cite[Proposition~4.11]{CL24}}). 
Hence, this provides alternative (and shorter) proofs for~{\cite[Theorem~1.4]{LM21}} and~{\cite[Theorem~1.1]{BJ24}}. 

The proof of Theorem~\ref{THM:C5-Kr}~\ref{THM:C5-Kr-1} is presented in Section~\ref{SEC:proof-C5-Kr}. 
The proof of Theorem~\ref{THM:C5-Kr}~\ref{THM:C5-Kr-2} is presented in Section~\ref{SEC:proof-C2r+1-C2r-1}.
In the next section, we present some definitions and preliminary results. 

\section{Preliminaries}\label{SEC:prelim}
Given a graph $G$ and a vertex $v\in V(G)$, the \textbf{neighborhood} of $v$ in $G$ is $N_{G}(v) \coloneqq \{u \in V(G) \colon uv\in G\}$. 
The degree of $v$ in $G$ is $d_{G}(v) \coloneqq |N_{G}(v)|$. 
Given integers $n \ge r\ge 2$, the Tur\'{a}n graph $T(n,r)$ is the balanced complete $r$-partite on $[n]$. 

The following definitions are motivated by the seminal works of Simonovits~\cite{SI68} and Andr\'{a}sfai--Erd\H{o}s--S\'{o}s~\cite{AES74}, as well as recent work~\cite{LMR23unif}. 
Let $Q$ be a graph and $\mathcal{F}$ be a family of graphs. 
Let $\mathfrak{H}$ be a hereditary\footnote{Here, hereditary means that if $H\in \mathfrak{H}$, then every subgraph of $H$ is also contained in $\mathfrak{H}$.} family of $\mathcal{F}$-free graphs. 
\begin{enumerate}[label=(\roman*)]
    \item We say $\mathcal{F}$ is \textbf{$Q$-edge-stable} with respect to $\mathfrak{H}$ if for every $\varepsilon>0$ there exist $\delta>0$ and $N_0$ such that every $\mathcal{F}$-free graph $G$ on $n \ge N_0$ vertices with $\mathrm{inj}(Q, \mathcal{H}) \ge (1-\delta) \cdot \mathrm{inj}(n,Q,\mathcal{F})$ is contained in $\mathfrak{H}$ after removing at most $\varepsilon n^r$ edges. 
    \item We say $\mathcal{F}$ is \textbf{$Q$-degree-stable} with respect to $\mathfrak{H}$ if there exist $\delta>0$ and $N_0$ such that every $\mathcal{F}$-free graph $G$ on $n \ge N_0$ vertices with $d_{Q}(\mathcal{H}) \ge (1-\delta) \cdot \frac{v(Q)\cdot \mathrm{inj}(n,Q,\mathcal{F})}{n}$ is contained in $\mathfrak{H}$. 
    \item We say $\mathcal{F}$ is \textbf{$Q$-vertex-extendable} with respect to $\mathfrak{H}$ if there exist $\delta>0$ and $N_0$ such that the following holds for every $\mathcal{F}$-free graph $G$ on $n \ge N_0$ vertices with $d_{Q}(\mathcal{H}) \ge (1-\delta) \cdot \frac{v(Q)\cdot \mathrm{inj}(n,Q,\mathcal{F})}{n} \colon$ 
    if $\mathcal{H}-v \in \mathfrak{H}$ for some $v \in V(\mathcal{H})$, then $\mathcal{H} \in \mathfrak{H}$. 
\end{enumerate}

The following result is a special case of~{\cite[Theorem~4.10]{CL24}}. 
\begin{theorem}[\cite{CL24}]\label{THM:general-generalized-Turan-a}
    Let $Q$ be a graph and $\mathcal{F}$ be a family of graphs satisfying $\pi(Q, \mathcal{F}) > 0$.  
    Let $\mathfrak{H}$ is a hereditary family of $\mathcal{F}$-free graphs. 
    Suppose that $\mathcal{F}$ is $Q$-edge-stable and $Q$-vertex-extendable with respect to $\mathfrak{H}$. 
    Then $\mathcal{F}$ is $Q$-degree-stable with respect to $\mathfrak{H}$. 
\end{theorem}
For every $r$, let $\mathfrak{K}_{r}$ denote the collection of all $K_{r}$-colorable graphs, and let $\mathfrak{C}_{r}$ denote the collection of all $C_{r}$-colorable graphs. 

The edge-stability of $\mathrm{ex}(n,C_5,K_{r+1})$ was established in~\cite{LM21}. 
\begin{theorem}[{\cite[Lemma~3.6]{LM21}}]\label{THM:stability-C5-Kr}
    For every $r\ge 3$, the graph $K_{r+1}$ is $C_5$-edge-stable with respect to the family $\mathfrak{K}_{r}$.  
\end{theorem}
The edge-stability of $\mathrm{ex}(n,C_{2r+1},C_{2r-1})$ was established in~\cite{BJ24}, though it can also be deduced from the proof in~\cite{GK23}. 
\begin{theorem}[\cite{GK23}~{\cite[Lemma~2.10]{BJ24}}]\label{THM:stability-C2r+1-C2r-1}
    For every $r\ge 3$, the graph $C_{2r-1}$ is $C_{2r+1}$-edge-stable with respect to the family $\mathfrak{C}_{2r+1}$.  
\end{theorem}

\section{Proof of Theorem~\ref{THM:C5-Kr}}\label{SEC:proof-C5-Kr}
\subsection{Proof of Theorem~\ref{THM:C5-Kr}~\ref{THM:C5-Kr-1}}\label{SEC:proof-C5-Kr}
We prove Theorem~\ref{THM:C5-Kr}~\ref{THM:C5-Kr-1} in this subsection. 
Note that by Theorems~\ref{THM:general-generalized-Turan-a} and~\ref{THM:stability-C5-Kr}, it suffices to establish the following result. 
\begin{proposition}\label{PROP:vtx-ext-C5-Kr}
    For every $r\ge 3$, the graph $K_{r+1}$ is $C_5$-vertex-extendable with respect to the family $\mathfrak{K}_{r}$.  
\end{proposition}
\begin{proof}[Proof of Proposition~\ref{PROP:vtx-ext-C5-Kr}]
    Fix $r\ge 3$. 
    Let $0 < \delta \ll \delta_1 \ll \delta_2 \ll \delta_3$ be sufficiently small, and let $n$ be a sufficiently large integer. 
    Let $G$ be an $n$-vertex $K_{r+1}$-free graph satisfying~\eqref{equ:THM:C5-Kr}.
    Suppose $v_{\ast} \in V(G)$ is a vertex such that $G-v_{\ast}$ is $K_{r}$-colorable. 
    Let $V\coloneqq V(G)\setminus \{v_{\ast}\}$. 
    Fix a homomorphism $\varphi \in \mathrm{Hom}(G,K_{r})$ and let $V_i \coloneqq \varphi^{-1}(i)$ for $i \in [r]$. 
    Let $x_i \coloneqq |V_i|/n$ for $i \in [r]$. 
    
    The following claim can be derived through some straightforward but tedious calculations (see~{\cite[Claim~3.10]{LM21}}). 
    \begin{claim}\label{CLAIM:proof-C5-xi}
        We have $\left|x_i - 1/r\right| \le \delta_1$ for every $i \in [r]$.
    \end{claim}
    The following claim follows from the assumption that $\delta_{C_5}(G) \ge (1-\delta)\cdot \frac{5\cdot \mathrm{inj}(n,C_5,K_{r+1})}{n}$ and some straightforward calculations.
    \begin{claim}\label{CLAIM:proof-C5-degree}
        For every $i \in [r]$ and $v \in V_i$, we have 
        \begin{align*}
            |N_{G}(v) \cap V_j| \ge x_j n - \delta_2 n
            \quad\text{for every } j \in [r]\setminus\{i\}. 
        \end{align*}
    \end{claim}
    Let $N_{i} \coloneqq N_{G}(v_{\ast}) \cap V_{i}$ for every $i \in [r]$. 
    The key claim in the proof is as follows. 
    \begin{claim}\label{CLAIM:proof-C5-v-ast-neighbor-a}
        There exists a unique $i_{\ast} \in [r]$ such that $|N_i| \le \delta_3 n$. 
    \end{claim}
    \begin{proof}[Proof of Claim~\ref{CLAIM:proof-C5-v-ast-neighbor-a}]
        First, suppose to the contrary that $\min_{i\in [r]}\left\{|N_i|\right\} > \delta_3 n$. 
        Then, for every $i \in [r]$, choose uniformly at random a vertex $u_i$ from $N_i$. 
        Let $\mathbf{A}$ denote the event that $\{v_{\ast}, u_1, \ldots, u_{r}\}$ induces a copy of $K_{r+1}$ in $G$. Note that $\mathbf{A}$ happens iff $u_i u_j \in G$ for all $1\le i < j \le r$. 
        So it follows from Claim~\ref{CLAIM:proof-C5-degree} and the Union Bound that 
        \begin{align*}
            \mathbb{P}[\mathbf{A}]
            \ge 1 - \sum_{1\le i < j \le r}\mathbb{P}[u_i u_j \not\in G]
            \ge 1 - \binom{r}{2} \cdot \frac{\delta_2}{\delta_3} 
            > 0,  
        \end{align*}
        which means that there exists an $r$-tuple $(u_1, \ldots, u_{r}) \in N_1 \times \cdots \times N_r$ such that $\{v_{\ast}, u_1, \ldots, u_{r}\}$ induces a copy of $K_{r+1}$ in $G$, a contradiction. 

        Now, suppose to the contrary that $|N_{j_1}| \le \delta_3 n$ and $|N_{j_2}| \le \delta_3 n$ for two distinct $j_1, j_2 \in [r]$. 
        By symmetry, we may assume that $\{j_1, j_2\} = \{1,2\}$. 
        Let $U_1 \coloneqq V_1 \cup \{v_{\ast}\}$ and $U_i \coloneqq V_i$ for $i \in [2,r]$. 
        Let $H$ denote the complete $r$-partite graph with parts $U_1, \ldots, U_r$. 
        Note that the difference between $d_{C_5, G}(v_{\ast})$ and $d_{C_5, H}(v_{\ast})$ satisfies 
        \begin{align*}
            d_{C_5, H}(v_{\ast}) - d_{C_5, G}(v_{\ast})
            & \ge 10 \left(\left(|U_2| - |N_2|\right) \cdot \sum_{S\in \binom{[3,r]}{3}} \prod_{i \in S}|U_{i}| - |N_1| \cdot n^3 \right) \\
            & \ge 10 \left(\left(\frac{n}{r}-\delta_2 n - \delta_3 n\right) \cdot \binom{r-2}{3}\left(\frac{n}{r}-\delta_2 n\right)^3 - \delta_3n \cdot n^3 \right) \\
            & \ge 10 \left(\binom{r-2}{3}\left(\frac{1}{2r}\right)^4 - \delta_3 \right)n^4. 
        \end{align*}
        On the other hand, 
        the difference between $d_{C_5, H}(v_{\ast})$ and $\frac{5}{n} \cdot \mathrm{inj}(n,C_5, K_{r+1}) = d_{C_5}(T(n,r))$ (by the result of Lidick\'{y}--Murphy~\cite{LM21}) satisfies 
        \begin{align*}
            \left|d_{C_5, H}(v_{\ast}) - \frac{5}{n} \cdot \mathrm{inj}(n,C_5, K_{r+1})\right| 
            & = \left|d_{C_5, H}(v_{\ast}) - d_{C_5}(T(n,r))\right| \\
            & \le 10 \cdot \sum_{i \in [2,r]}\left(|U_2| - \frac{n}{r}\right) \cdot n^3 
             \le 10 \cdot 5 \cdot \delta_2 n \cdot n^3 
            = 50 \delta_2 n^4. 
        \end{align*}
        Combining these two inequalities, we obtain 
        \begin{align*}
            \left|d_{C_5, G}(v_{\ast}) - \frac{5}{n} \cdot \mathrm{inj}(n,C_5, K_{r+1})\right|
            & \ge 10 \left(\binom{r-2}{3}\left(\frac{1}{2r}\right)^4 - \delta_3 \right)n^4 - 50 \delta_2 n^4 \\
            & > \delta \cdot \frac{5}{n} \cdot \mathrm{inj}(n,C_5, K_{r+1}), 
        \end{align*}
        a contradiction. 
    \end{proof}
    \begin{claim}\label{CLAIM:proof-C5-v-ast-neighbor-b}
        We have $N_{G}(v_{\ast}) \cap V_{i_{\ast}} = \emptyset$, where $i_{\ast} \in [r]$ is the index guaranteed by Claim~\ref{CLAIM:proof-C5-v-ast-neighbor-a}. 
    \end{claim}
    \begin{proof}[Proof of Claim~\ref{CLAIM:proof-C5-v-ast-neighbor-a}]
        By symmetry, we may assume that $i_{\ast} =1$. 
        Suppose to the contrary that there exists a vertex $v_1 \in N_{G}(v_{\ast}) \cap V_{1}$.
        Let $\hat{N}_{i} \coloneqq N_{G}(v_{\ast}) \cap N_{G}(v_{1}) \cap V_{i}$ for every $i \in [2,r]$. 
        It follows from Claims~\ref{CLAIM:proof-C5-degree} and~\ref{CLAIM:proof-C5-v-ast-neighbor-a} that $|\hat{N}_i| \ge \delta_3 n - \delta_2 n \ge \delta_3 n/2$. 
        Then similar to the proof of Claim~\ref{CLAIM:proof-C5-v-ast-neighbor-a}, there exists an $(r-1)$-tuple $(u_2, \ldots, u_{r}) \in \hat{N}_2 \times \cdots \times \hat{N}_r$ such that $\{v_{\ast}, v_1, u_2, \ldots, u_{r}\}$ induces a copy of $K_{r+1}$ in $G$, a contradiction. 
    \end{proof}
Claim~\ref{CLAIM:proof-C5-v-ast-neighbor-b} completes the proof of Proposition~\ref{PROP:vtx-ext-C5-Kr}. 
\end{proof}

\subsection{Proof of Theorem~\ref{THM:C5-Kr}~\ref{THM:C5-Kr-2}}\label{SEC:proof-C2r+1-C2r-1}
We prove Theorem~\ref{THM:C5-Kr}~\ref{THM:C5-Kr-2} in this subsection. 
Note that by Theorems~\ref{THM:general-generalized-Turan-a} and~\ref{THM:stability-C2r+1-C2r-1}, it suffices to establish the following result. 
\begin{proposition}\label{PROP:vtx-ext-C2r+1-C2r-1}
    For every $r\ge 3$, the graph $C_{2r-1}$ is $C_{2r+1}$-vertex-extendable with respect to the family $\mathfrak{C}_{2r+1}$.     
\end{proposition}
\begin{proof}[Proof of Proposition~\ref{PROP:vtx-ext-C2r+1-C2r-1}]
    Fix $r \ge 3$. 
    Let $0 < \delta \ll \delta_1 \ll \delta_2 \ll \delta_3$ be sufficiently small, and let $n$ be a sufficiently large integer. 
    Let $G$ be an $n$-vertex $C_{2r-1}$-free graph satisfying~\eqref{equ:THM:C2r+1-C2r-1}.
    Suppose $v_{\ast} \in V(G)$ is a vertex such that $G-v_{\ast}$ is $C_{2r+1}$-colorable. 
    Let $V\coloneqq V(G)\setminus \{v_{\ast}\}$. 
    Fix a homomorphism $\varphi \in \mathrm{Hom}(G,C_{2r+1})$ and let $V_i \coloneqq \varphi^{-1}(i)$ for $i \in [2r+1]$. 
    Let $x_i \coloneqq |V_i|/n$ for $i \in [2r+1]$. 
    Simple calculations using Maclaurin's inequality and the assumption that $\delta_{C_{2r+1}}(G) \ge (1-\delta) \cdot \frac{(2r+1)\cdot \mathrm{inj}(n,C_{2r+1},C_{2r-1})}{n}$ yield the following claim. 
    \begin{claim}\label{CLAIM:proof-C2r+1-xi}
        The following statements hold. 
        \begin{enumerate}[label=(\roman*)]
            \item\label{CLAIM:proof-C2r+1-xi-1} We have $\left|x_i - \frac{1}{2r+1}\right| \le \delta_1$ for every $i \in [2r+1]$. 
            \item\label{CLAIM:proof-C2r+1-xi-2} For every $i \in [2r+1]$ and $v \in V_i$, we have 
            \begin{align*}
                \max\left\{|V_{i-1} \setminus N_{G}(v)|,\  |V_{i+1} \setminus N_{G}(v)|\right\} 
                \le \delta_2 n. 
            \end{align*}
            Here, the indices are taken modulo $2r+1$.    
        \end{enumerate}
    \end{claim}
    Let $N_{i} \coloneqq N_{G}(v_{\ast}) \cap V_{i}$ for every $i \in [2r+1]$. 
    Let $I \coloneqq \left\{i \in [2r+1] \colon |N_{i}| \ge \delta_3 n\right\}$ and $J \coloneqq \left\{i \in [2r+1] \colon N_{i} \neq \emptyset\right\}$, noting that $I \subseteq J$. 
    The key claim in the proof is as follows. 
    \begin{claim}\label{CLAIM:proof-C2r+1-v-ast-neighbor-a}
        The following statements hold. 
        \begin{enumerate}[label=(\roman*)]
            \item\label{CLAIM:proof-C2r+1-v-ast-neighbor-a-1} We have $|I| \ge 2$. 
            \item\label{CLAIM:proof-C2r+1-v-ast-neighbor-a-2} For every $j \in J$ and for every $i \in I$, we have $|i - j| = 2$. 
        \end{enumerate}
        Consequently, $I = J = \{i_{\ast}, i_{\ast}+2\}$ for some $i_{\ast} \in [2r+1]$. 
    \end{claim}
    \begin{proof}[Proof of Claim~\ref{CLAIM:proof-C2r+1-v-ast-neighbor-a}]
        First, suppose to the contrary that $|I| \le 1$. 
        By symmetry, we may assume that $I = \{1\}$. 
        Since at most one of the two neighbors of $v_{\ast}$ in any copy of $C_{2r+1}$ containing $v_{\ast}$ can lie in $N_1$, we obtain  
        \begin{align*}
            d_{C_{2r+1},G}(v_{\ast})
            & \le 2(2r+1) \cdot \left(\sum_{i\in [2,2r+1]}|N_1||N_i|+ \sum_{\{i,j\}\subseteq [2,2r+1]}|N_i||N_j|\right) \cdot n^{2r-1} \\
            & \le 2(2r+1)\cdot \left(2r\delta_3 n^2 + \binom{2r}{2}\delta_3^2 n^2\right) \cdot n^{2r-1} 
            < 8r^2(2r+1)\delta_3 n^{2r+1}
            < \delta_{C_{2r+1}}(G), 
        \end{align*}
        a contradiction. 
        Therefore, we have $|I| \ge 2$. 

        Next, we prove~\ref{CLAIM:proof-C2r+1-v-ast-neighbor-a-2}. 
        Suppose to the contrary that there exists $(j_1, j_2) \in J \times I$ such that $|j_1 - j_2| \neq 2 \pmod{2r+1}$. 
        By symmetry, we may assume that $j_2 = 1$. 
        Notice that either $j_1 - 1$ or $1- j_1$ is odd. 
        So, by symmetry, we may assume that $j_1 -1$ is odd. This means that $j_1$ (which is at most $2r+1-3 = 2(r-1)$) is even, and consequently, $2r-j_1$ is even. 
        
        Let $\ell \coloneqq \frac{2(r-1)-j_1+2}{2} \ge 1$. We choose $2(r-1)$ vertices according to the following rules. 
        \begin{itemize}
            \item Choose uniformly at random $\ell$ vertices $u_1, \ldots, u_{\ell}$ from $N_1$. 
            \item Choose uniformly at random $\ell$ vertices $v_1, \ldots, v_{\ell}$ from $V_2$.
            \item For each $i \in [3,j_1-1]$, choose uniformly at random a vertex $w_i$ from $V_i$.
            \item Choose an arbitrary vertex $w_{j_1}$ from $N_{j_1}$.
        \end{itemize}
        Let $\mathbf{A}$ denote the event that $\{v_{\ast}, u_1, \ldots, u_{\ell}, v_1, \ldots, v_{\ell}, w_3, \ldots, w_{j_1}\}$ spans a copy of $C_{2r-1}$ in $G$.
        Note that $\mathbf{A}$ occurs if  
        \begin{itemize}
            \item $u_i v_i \in G$  for all $i \in [\ell]$, 
            \item $u_{i+1}v_i \in G$ for all $i \in [\ell-1]$, 
            \item $v_{\ell} w_3 \in G$, and 
            \item $w_iw_{i+1} \in G$ for $i \in [3,j_1-1]$.  
        \end{itemize}
        Similar to the proof of Claim~\ref{CLAIM:proof-C5-v-ast-neighbor-a}, it follows from Claim~\ref{CLAIM:proof-C2r+1-xi} and the Union Bound that 
        \begin{align*}
            \mathbb{P}[\mathbf{A}]
            \ge 1 - (2r-3) \cdot \frac{\delta_2}{\delta_3} 
            > 0,  
        \end{align*}
        which means that there exists a selection of vertices $\{v_{\ast}, u_1, \ldots, u_{\ell}, v_1, \ldots, v_{\ell}, w_3, \ldots, w_{j_1}\}$ that spans a copy of $C_{2r-1}$ in $G$, a contradiction.
    \end{proof}
    Claim~\ref{CLAIM:proof-C2r+1-v-ast-neighbor-a} completes the proof of Proposition~\ref{PROP:vtx-ext-C2r+1-C2r-1}. 
\end{proof}
\bibliographystyle{alpha}
\bibliography{CountCycle}

\newcommand{\etalchar}[1]{$^{#1}$}
\begin{thebibliography}{HHK{\etalchar{+}}13}

\bibitem[AES74]{AES74}
B.~Andr\'{a}sfai, P.~Erd{\H o}s, and V.~T. S\'{o}s.
\newblock On the connection between chromatic number, maximal clique and minimal degree of a graph.
\newblock {\em Discrete Math.}, 8:205--218, 1974.

\bibitem[AS16]{AS16}
Noga Alon and Clara Shikhelman.
\newblock Many {$T$} copies in {$H$}-free graphs.
\newblock {\em J. Combin. Theory Ser. B}, 121:146--172, 2016.

\bibitem[BJ24]{BJ24}
Csongor Beke and Oliver Janzer.
\newblock On the {G}eneralized {T}ur\'an {P}roblem for {O}dd {C}ycles.
\newblock {\em SIAM J. Discrete Math.}, 38(3):2416--2428, 2024.

\bibitem[CL24]{CL24}
Wanfang Chen and Xizhi Liu.
\newblock Strong stability from vertex-extendability and applications in generalized {T}ur{\' a}n problems.
\newblock {\em arXiv preprint arXiv:2406.05748}, 2024.

\bibitem[Erd62]{Erd62gen}
P.~Erd{\H o}s.
\newblock On the number of complete subgraphs contained in certain graphs.
\newblock {\em Magyar Tud. Akad. Mat. Kutat\'o{} Int. K\"ozl.}, 7:459--464, 1962.

\bibitem[Erd84]{Erd84}
Paul Erd\H{o}s.
\newblock On some problems in graph theory, combinatorial analysis and combinatorial number theory.
\newblock In {\em Graph theory and combinatorics ({C}ambridge, 1983)}, pages 1--17. Academic Press, London, 1984.

\bibitem[GK23]{GK23}
D{\'a}niel Gerbner and Hilal~Hama Karim.
\newblock Stability from graph symmetrization arguments in generalized {T}ur{\' a}n problems.
\newblock {\em arXiv preprint arXiv:2303.17718}, 2023.

\bibitem[Grz12]{Gre12}
Andrzej Grzesik.
\newblock On the maximum number of five-cycles in a triangle-free graph.
\newblock {\em J. Combin. Theory Ser. B}, 102(5):1061--1066, 2012.

\bibitem[HHK{\etalchar{+}}13]{HHKNR13}
Hamed Hatami, Jan Hladk\'{y}, Daniel Kr\'{a}\v{l}, Serguei Norine, and Alexander Razborov.
\newblock On the number of pentagons in triangle-free graphs.
\newblock {\em J. Combin. Theory Ser. A}, 120(3):722--732, 2013.

\bibitem[KNS64]{KNS64}
Gyula Katona, Tibor Nemetz, and Mikl\'{o}s Simonovits.
\newblock On a problem of {T}ur\'{a}n in the theory of graphs.
\newblock {\em Mat. Lapok}, 15:228--238, 1964.

\bibitem[LM21]{LM21}
Bernard Lidick\'{y} and Kyle Murphy.
\newblock Maximizing five-cycles in {$K_r$}-free graphs.
\newblock {\em European J. Combin.}, 97:Paper No. 103367, 29, 2021.

\bibitem[LMR23]{LMR23unif}
Xizhi Liu, Dhruv Mubayi, and Christian Reiher.
\newblock A unified approach to hypergraph stability.
\newblock {\em J. Combin. Theory Ser. B}, 158(part 2):36--62, 2023.

\bibitem[LP18]{LP18}
Bernard Lidick\'{y} and Florian Pfender.
\newblock Pentagons in triangle-free graphs.
\newblock {\em European J. Combin.}, 74:85--89, 2018.

\bibitem[Raz07]{Raz07}
Alexander~A. Razborov.
\newblock Flag algebras.
\newblock {\em J. Symbolic Logic}, 72(4):1239--1282, 2007.

\bibitem[Sim68]{SI68}
M.~Simonovits.
\newblock A method for solving extremal problems in graph theory, stability problems.
\newblock In {\em Theory of {G}raphs ({P}roc. {C}olloq., {T}ihany, 1966)}, pages 279--319. Academic Press, New York, 1968.

\bibitem[Tur41]{TU41}
Paul Tur{\'a}n.
\newblock On an extermal problem in graph theory.
\newblock {\em Mat. Fiz. Lapok}, 48:436--452, 1941.

\end{thebibliography}
\end{document}